\documentclass[a4paper]{easychair}
\usepackage[utf8]{inputenc}
\usepackage{txfonts}
\usepackage{tikz}
\usepackage{booktabs}
\usepackage{multicol}

\usepackage{hyperref}
\hypersetup{
  bookmarks=true,
  pdftitle={Exploring Steinitz-Rademacher polyhedra: a challenge for automated reasoning tools},
  pdfkeywords={polyhedra, automated theorem provers, discrete mathematics, theory exploration},
  pdfauthor={Jesse Alama},
  pdfsubject={A paper presented at the 8th International Workshop on the Implementation of Logics (http://www.eprover.org/EVENTS/IWIL-2010/iwil-2010.html)},
  colorlinks=true,
  linkcolor=black,
  citecolor=black,
}

\begin{document}

\def\book#1{\emph{#1}}
\def\term#1{\textbf{#1}}
\def\systemname#1{\textsf{#1}}
\def\macefour{\systemname{mace4}}
\def\spass{\systemname{SPASS}}
\def\paradox{\systemname{paradox}}
\def\equinox{\systemname{equinox}}
\def\plantri{\systemname{plantri}}
\def\sem{\systemname{sem}}
\def\SR{\mathsf{SR}}
\def\SRcard#1#2#3{\SR_{#1,#2,#3}}
\def\SRext{\mathsf{SR}_{\mathrm{ext}}}
\def\Mod#1{\mathrm{Mod}(#1)}

\let\phi=\varphi

\newtheorem{theorem}{Theorem}
\newtheorem{definition}[theorem]{Definition}

\title{Exploring Steinitz-Rademacher polyhedra: a challenge for automated reasoning tools}
\titlerunning{Exploring Steinitz-Rademacher polyhedra}
\author{Jesse Alama\thanks{Partially supported by the ESF
research project \emph{Dialogical Foundations of Semantics} within
the ESF Eurocores program \emph{LogICCC} (funded by the Portuguese
Science Foundation, FCT LogICCC/0001/2007).}\\Center for Artificial Intelligence\\Department of Computer Science, Faculty of Science and Technology\\New University of Lisbon, Portugal\\\url{j.alama@fct.unl.pt}}
\authorrunning{J.~Alama}
\maketitle

\begin{abstract}
  This note reports on some experiments, using a handful of standard
  automated reasoning tools, for exploring Steinitz-Rademacher
  polyhedra, which are models of a certain first-order theory of
  incidence structures.  This theory and its models, even simple ones,
  presents significant, geometrically fascinating challenges for
  automated reasoning tools.
\end{abstract}

\section{Introduction}
\label{sec:introduction}

This note reports on some experiments, using a handful of standard
automated reasoning tools, for exploring a certain first-order theory
of three-dimensional polyhedra.  Polyhedra are understood here as
combinatorial objects, rather than as, say, certain kinds of
structures in $\mathbf{R}^{3}$.

Specifically, the polyhedra considered here are Steinitz-Rademacher
polyhedra (they will be defined in Section~2).  As first-order
structures, these polyhedra are directed graphs with three
sorts---vertices, edges, and faces---satisfying some intuitive
geometric principles shared by ``everyday'' three-dimensional
polyhedra.

Restricting ourselves to first-order logic makes it possible to take
advantage of automated reasoning tools that work well for FOL, but
our restriction comes with a price: many interesting features of
graphs, such as connectivity or the property of satisfying Euler's
formula, for example, cannot be expressed in FOL.\footnote{FOL's
  failure to express these and other properties of graphs holds even
  when one restricts attention to finite structures; such results can
  be shown using Ehrenfeucht-Fraïssé
  games~\cite{libkin2004}.}  
Nonetheless, it is not clear that FOL's lack of expressive power
precludes the possibility of learning something about polyhedra.  We
believe that the preliminary results discussed here do have
interesting mathematical content.

An investigation of polyhedra with automated reasoners is valuable for
two domains.  First, the investigation is valuable for mathematics,
because automated reasoners offer the possibility, in certain
contexts, of a more objective investigation than one carried out by
entirely by humans, who are prone to make subtle flaws when reasoning
about space.  Second, the investigation is valuable for automated
reasoning, because working with polyhedra---even small
ones---naturally leads to challenging problems, as we shall see.


There do not appear to be many explorations of polyhedra using
automated reasoning tools.  Much has been done on enumerating
polyhedra (or related combinatorial structures, such as planar graphs)
using mathematical rather than logical techniques; one such system is
the highly efficient \plantri~\cite{brinkmann-mckay2003}).  Within the
realm of automated reasoning, L.~Schewe has used SAT solvers to
investigate realizability of abstract simplicial
complexes~\cite{schewe-matroid}.


\section{Steinitz-Rademacher polyhedra}

E.~Steinitz asked~\cite{steinitz-rademacher1976}: when can a
combinatorially given polyhedron---a collection of abstract vertices,
edges, and faces, with an incidence relation among these
polytopes---be realized as a three-dimensional convex polyhedron in
$\mathbf{R}^{3}$?\footnote{The answer, known as Steinitz's
  theorem~\cite{grunbaum2007}, says that a directed graph $g$ is
  isomorphic to the $1$-skeleton of a real convex three-dimensional
  polyhedron iff $g$ is planar and three-connected.}  As part of his
initial investigation Steinitz formulated a basic theory of
combinatorial polyhedra that forms the basis for our investigation as
well.

First, we specify an unsorted first-order signature:
\begin{definition}
  Let $\pi$ be the first-order signature (with equality) containing
  three unary predicates $V$ (for ``vertex''), $E$ (for ``edge''), and
  $F$ (for ``face''), and one binary relation $I$ (for incidence).
\end{definition}
The signature $\pi$ provides a rudimentary language for talking about
three-dimensional polyhedra.  Alternatively, one can view
$\pi$-structures simply as directed graphs whose nodes can be painted
with one of three ``colors'' $V$, $E$, and $F$.


\begin{definition} The theory $\SR$ of Steinitz-Rademacher
  polyhedra consists of the statements:
  \begin{multicols}{2}
  \begin{itemize}
    \setlength{\itemsep}{0pt}
    \setlength{\parsep}{0pt}
  \item there are vertices, edges, and faces;
  \item every element is a vertex, edge, or a face;
  \item $I$ is symmetric;
  \item no two vertices are incident, and the same goes for edges and faces;
  \item if $V(v)$, $E(e)$, $F(f)$,
    $I(v,e)$ and $I(e,f)$, then $I(v,f)$;
  \item every edge is incident with exactly two vertices;
  \item every edge is incident with exactly two faces;
  \item $V(v)$, $F(f)$ and $I(v,f)$ imply that there are exactly two
    edges incident with both $v$ and $f$; and
  \item every vertex and every face is incident with at least one
    other element.
  \end{itemize}
  \end{multicols}
\end{definition}
These conditions are expressible as first-order
$\pi$-sentences.
\begin{definition}
  A \term{Steinitz-Rademacher polyhedron} (or \term{SR-polyhedron}) is
  a model of the theory $\SR$.
\end{definition}

Some questions that we would like to address about SR-polyhedra, in
this note, are:
\begin{enumerate}
\setlength{\itemsep}{0pt}
\setlength{\parsep}{0pt}
\item $\SR$ is consistent (just think of, say, a tetrahedron).
  What is the smallest model?\footnote{There is no largest finite
    SR-polyhedron.  Consider, for example, the sequence $\langle P_{n}
    \mid n \geq 3 \rangle$ of pyramids, each $P_{n}$ characterized by
    an $n$-gon $B_{n}$ for its base and a single point
    ``above'' the base incident with each of the vertices of $B_{n}$.
    Each $P_{n}$ is evidently an SR-polyhedron has cardinality
    $(n+1)+2n +(n+1) = 4n+2$.}
\item For which natural numbers $k$ is $\SR$
  $k$-categorical?\footnote{$\SR$ is not $\lambda$-categorical for any
    infinite cardinal $\lambda$.  Consider, for example, the
    ``tessellation'' $M_{\lambda}^{3}$ having $\lambda$ vertices,
    $\lambda$ edges, and $\lambda$ faces, each of which is a triangle
    that meets three other triangles along its three edges, \textit{ad
      $\lambda$-infinitum}; and $M_{\lambda}^{4}$, which is a
    ``tessellation'' like $M_{\lambda}^{3}$ except that each face of
    $M_{\lambda}^{4}$ has four edges rather than three.  In
    $M_{\lambda}^{3}$ the $\pi$-sentence ``every face is a triangle''
    holds, but by construction it fails in $M_{\lambda}^{4}$. Or,
    continuing the discussion of the previous footnote, consider a
    tetrahedron and a cube, each of whose edges and faces contains
    $\lambda$-many vertices (but each having their usual finite number
    of edges and faces).}
\item How hard is it to ``recover'' (that is, produce as models of
  $\SR$) well-known polyhedra (e.g., the platonic solids) as models of
  $\SR$?
\item Can we discover unusual or unexpected SR-polyhedra?
\end{enumerate}
The rest of the paper takes up these questions.

Question~1 is also one of the first questions raised by Steinitz and Rademacher~\cite{steinitz-rademacher1976}.
\begin{theorem}\label{card-6-smallest}
  There is a Steinitz-Rademacher polyhedron of cardinality 6, but none
  of smaller cardinality.
\end{theorem}
\begin{proof}
  This result is readily confirmed with the help of a first-order
  model generation tool (e.g., \macefour{}~\cite{mace4}
  or \paradox{}/\equinox{}~\cite{ClaessenS03}).  A refutational
  theorem prover can then show that $\SR$, extended with an axiom
  saying that there are at most five elements, is inconsistent.
\end{proof}
Figure 1 illustrates this smallest SR-polyhedron, $M_{6}$, with two
vertices, two edges, and two faces; the two edges are the upper and
lower semicircular arcs, and the two faces are the inside and outside
of the circle.
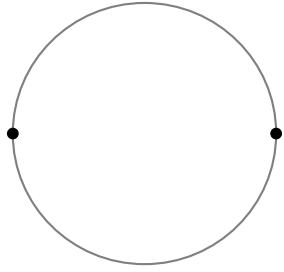
\begin{figure}[h]
  \label{card-6-sr-poly}
  \centering
  \begin{tikzpicture}
    \draw [gray,thick] (0,0) circle (4.5em);
    \filldraw[black] (-1*4.5em,0) circle (2pt);
    \filldraw[black] (1*4.5em,0) circle (2pt);
  \end{tikzpicture}
  \caption{$M_{6}$: the smallest SR-polyhedron, with six elements: two vertices, two edges, two faces}
\end{figure}
$M_{6}$ has the curious feature that every vertex is incident with
every edge and with every face.  Depending on one's view about
polyhedra, $M_{6}$ might be a positive solution to Question~4; see
Section~\ref{future-work}.

\begin{theorem}
  Up to isomorphism, $M_{6}$ is the only SR-polyhedron of cardinality
  six.
\end{theorem}
\begin{proof}
  Each of the 28 triples of natural numbers $(N_{0},N_{1},N_{2})$ that
  sum to $6$ gives rise to an extension $\SRcard{N_{0}}{N_{1}}{N_{2}}$
  of $\SR$ obtained by adding axioms saying that there are exactly
  $N_{0}$ vertices, $N_{1}$ edges, and $N_{2}$ faces.  All but one of
  these 28 theories---namely, $\SRcard{2}{2}{2}$---are inconsistent;
  this can be shown by applying a standard refutational theorem prover
  to the 27 theories different from $\SRcard{2}{2}{2}$.

  To show that $M_{6}$ is, up to isomorphism, the only SR-polyhedron
  with exactly $2$ vertices, $2$ edges, and $2$ faces, consider the
  extension of $\SR$ by the formula
  \[
  \phi \coloneqq \exists x_{0},\dots,x_{5}
  \left [
    \begin{array}{c}
      V(x_{0}) \wedge V(x_{1}) \wedge x_{0} \neq x_{1} \wedge \forall x \left ( V(x) \rightarrow \left ( x = x_{0} \vee x =x_{1} \right ) \right )\\
      \wedge\\
       E(x_{2}) \wedge E(x_{3}) \wedge x_{2} \neq x_{3} \wedge \forall x \left ( E(x) \rightarrow \left ( x = x_{2} \vee x =x_{3} \right ) \right )\\
      \wedge\\
      F(x_{4}) \wedge F(x_{5}) \wedge x_{4} \neq x_{5} \wedge \forall x \left ( F(x) \rightarrow \left ( x = x_{4} \vee x =x_{5} \right ) \right )\\
      \wedge\\
      \neg \left ( I(x_{0},x_{2}) \wedge I(x_{1},x_{2}) \wedge I(x_{0},x_{3}) \wedge I(x_{1},x_{3}) \wedge I(x_{2},x_{4}) \wedge I(x_{3},x_{4}) \wedge I(x_{2},x_{5}) \wedge I(x_{3},x_{5}) \right )
    \end{array}
  \right ]
  \]
  The first three bundles of conjunctions in the matrix of $\phi$
  express the cardinality of the sets of vertices, edges, and faces.
  The final conjunction expresses the essential incidence relations
  holding among the elements of $M_{6}$, when one labels the vertices
  $x_{0}$ and $x_{1}$, the edges $x_{2}$ and $x_{3}$, and the faces
  $x_{4}$ and $x_{5}$; it is negated because we are trying to find a
  model that is \emph{unlike} $M_{6}$.  One then shows, with, e.g.,
  \macefour{}, that $\SR \cup \{ \phi \}$ is unsatisfiable.
\end{proof}
\begin{theorem}
  There are at least two SR-polyhedron of cardinality~8.
\end{theorem}
One such cardinality 8 SR-polyhedron is depicted in
Figure~\ref{card-8-sr-poly}.  The two vertices are clear; the three
edges are the upper and lower semicircles, plus the straight line
segment joining the two vertices.  The three faces are: the exterior
of the circle, incident with the upper and lower semicircular arcs;
and the two semicircles, each incident with one of the semicircular
arcs and both incident with the straight line segment.  Another is its
dual, $M_{8}^{d}$, which is obtained by exchanging the vertices and
faces of $M_{8}$ (notice that $\SR$ is invariant under this exchange).
\begin{figure}[h]
  \label{card-8-sr-poly}
  \centering
  \begin{tikzpicture}
    \draw [gray,thick] (0,0) circle (4.5em);
    \draw [gray,thick] (-1*4.5em,0) -- (1*4.5em,0);
    \filldraw[black] (-1*4.5em,0) circle (2pt);
    \filldraw[black] (1*4.5em,0) circle (2pt);
  \end{tikzpicture}
  \caption{An SR-polyhedron with eight elements: two vertices, three edges, three faces.}
\end{figure}
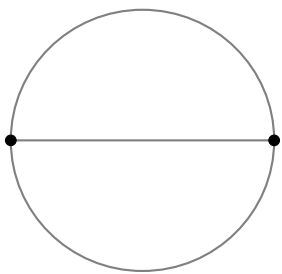

Returning to Question~4, as with, $M_{6}$, the SR-polyhedron $M_{8}$
and its dual $M_{8}^{d}$ could be regarded as an unusual or unexpected
model of $\SR$.  Although $\SR$ does admit curious ``polyhedra'', it
should be clear that one can take nearly any familiar polyhedron, such
as the platonic solids, as SR-polyhedra.

\begin{table}[h]
  \centering
  \begin{tabular}{l|r|r|r|r}
    \bf{Solid} & \bf{Num. Vertices} & \bf{Num. Edges} & \bf{Num. Faces} & \bf{Cardinality}\\
    \hline
    Tetrahedron    & 4        & 6     & 4     & 14\\
    Cube           & 8        & 12    & 6     & 26\\
    Octahedron     & 6        & 12    & 8     & 26\\
    Dodecahedron   & 20       & 30    & 12    & 62\\
    Icosahedron    & 12       & 30    & 20    & 62\\
  \end{tabular}
  \caption{Cardinal numbers for the platonic solids}
\end{table}
Notice that, since they are duals, the cube and the octahedron, as
well as the dodecahedron and the icosahedron, show that $\SR$ is
neither $26$- nor $62$-categorical.  We now have the modest beginnings
of an answer to Question~2: for $k = 6$ we have that $\SR$ is
$k$-categorical, and for $k =8$, $26$, and $62$ we know that $\SR$ is
not $k$-categorical, by duality.

Although the tetrahedron can be recovered as an SR-polyhedron (by,
e.g., \paradox{}), the remaining platonic solids seem to lie
tantalizingly beyond the scope of current automated reasoning tools: a
very large amount of time indeed seems to be required to automatically
generate these solids. It would be interesting to determine whether
the cube and its dual the icosahedron are the only SR-polyhedra of
cardinality 26.

For lack of space the investigation has to be cut short here.

\section{Future work}\label{future-work}

The range of Steinitz-Rademacher polyhedra is arguably too large: if
the tetrahedron is the ``simplest'' three-dimensional polyhedron, then
the six-element SR-polyhedron $M_{6}$ in Figure~\ref{card-6-sr-poly}
and other SR-polyhedra of cardinality less than that of the
tetrahedron, such as $M_{8}$ in Figure~\ref{card-8-sr-poly}, show that
$\SR$ alone lacks sufficient geometric content and needs to be
extended by principles that rule out such models.  Extensionality is a
natural candidate.  $M_{6}$, for example, has two distinct vertices
that share the same edges and faces, two distinct edges that share the
same two vertices and faces, and two distinct faces that share the
same vertices and edges.  The eight-element model $M_{8}$ is similar.
Intuitively, the incidence relation between polytopes is extensional:
if two polytopes $p$ and $q$ are incident with the
same set of polytopes, then $p = q$.

Adding extensionality to $\SR$ yields a new theory $\SRext$ of extensional
SR-polyhedra.
\paradox{} can recover the tetrahedron as the smallest model of
$\SRext$, thereby answering Question~1 and part of Question~3 for the
new theory.  However, extensionality raises a high computational
hurdle.  Concerning Questions~2 and~3, the search for models of
$\SRext$ has so far not been successful beyond the aforementioned
recovery of the tetrahedron.  Question~4, about $k$-categoricity, is
an even greater challenge for $\SRext$ than it was for $\SR$ and is
more intriguing because its models are more geometrically
intuitive.

The signature $\pi$ of SR-polyhedra is unsorted: $V$, $E$, $F$, and
$I$ are relations that could hold for arbitrary elements in a
$\pi$-structure.  It is likely that tools for sorted FOL, such as
\spass{}~\cite{spass}, would be more effective than the unsorted tools
used so far.  Constraint techniques for model generation, such as
those behind \sem{}~\cite{zhang-zhang1995}, ought also to be
evaluated.

\section{Conclusion}

Combinatorial polyhedra abstract away from positions in space and
regard polyhedra as incidence structures.  Steinitz-Rademacher
polyhedra are one such kind of polyhedra axiomatized
by an intuitive first-order theory.  We have proposed a handful of basic
questions about these polyhedra and shown that automated reasoners can
tackle some of them with verve, though others remain only partially
answered.  Even small Steinitz-Rademacher polyhedra present
significant challenges for automated reasoners; dealing with larger
ones will likely require new techniques.  We thus urge combinatorial
polyhedra as a tantalizingly fertile source of challenging automated
reasoning problems.

\bibliographystyle{plain}
\bibliography{alama}

\begin{thebibliography}{1}

\bibitem{brinkmann-mckay2003}
G.~Brinkmann and B.~D. McKay.
\newblock Fast generation of planar graphs.
\newblock {\em MATCH Comm. in Comp. Chem.}, 58:323--357, 2003.

\bibitem{ClaessenS03}
K.~Claessen and N.~Sorensson.
\newblock {New Techniques that Improve MACE-style Finite Model Finding}.
\newblock In P.~Baumgartner and C.~Fermueller, editors, {\em {Proc. of the
  CADE-19 Workshop: Model Computation: Principles, Algorithms, Applications}},
  2003.

\bibitem{grunbaum2007}
B.~Gr{\"u}nbaum.
\newblock Graphs of polyhedra; polyhedra as graphs.
\newblock {\em Discrete Mathematics}, 307:445--463, 2007.

\bibitem{libkin2004}
L.~Libkin.
\newblock {\em Elements of Finite Model Theory}.
\newblock Texts in Theoretical Computer Science. Springer, 2004.

\bibitem{mace4}
W.~McCune.
\newblock {Prover9} and {Mace4}.
\newblock \url{http://www.cs.unm.edu/~mccune/prover9/}.

\bibitem{schewe-matroid}
L.~Schewe.
\newblock Generation of oriented matroids using satisfiability solvers.
\newblock In A.~Iglesias and N.~Takayama, editors, {\em Proc. of the 2nd Int.
  Conf. on Math. Soft.}, volume 4151 of {\em Lecture Notes in Computer
  Science}. Springer, 2006.

\bibitem{steinitz-rademacher1976}
E.~Steinitz and H.~Rademacher.
\newblock {\em Vorlesungen {\"u}ber die Theorie der Polyeder unter Einschluss
  der Elemente der Topologie}.
\newblock Springer, 1976.
\newblock Reprint of the original 1934 edition.

\bibitem{spass}
C.~Weidenbach, D.~Dimova, A.~Fietzke, R.~Kumar, M.~Suda, and P.~Wischnewski.
\newblock \spass{} version 3.5.
\newblock In R.~Schmidt, editor, {\em Automated Deduction: CADE 22}, volume
  5663 of {\em Lecture Notes in Computer Science}, pages 140--145. Springer,
  2009.

\bibitem{zhang-zhang1995}
J.~Zhang and H.~Zhang.
\newblock {SEM}: A system for enumerating models.
\newblock In C.~P. Mellish, editor, {\em Proceedings of the Fourteenth
  International Joint Conference on Artificial Intelligence}, pages 298--303,
  1995.

\end{thebibliography}
\end{document}